\documentclass{amsart}
\usepackage{amssymb}
\usepackage{amsmath}
\usepackage{color,epsfig}
\usepackage{graphicx}
\usepackage{mathdots}

\usepackage[dvipdfm,colorlinks,backref,pagebackref,hypertexnames=false,linkcolor=blue,urlcolor=green,citecolor=red]{hyperref}

\newcommand{\N}{{\mathbb{N}}}
\newcommand{\C}{{\mathbb{C}}}
\newcommand{\F}{{\mathbb{F}}}
\newcommand{\T}{{\mathbb{T}}}
\newcommand{\A}{{\mathbb{A}}}
\newcommand{\B}{{\mathbb{B}}}

\newcommand{\FF}{{\mathfrak{F}}}

\newcommand{\im}{{\rm im}}
\newcommand{\chr}{{\rm char}}
\newcommand{\Stab}{{\rm Stab}}
\newcommand\Irr{{\rm Irr}}
\newcommand{\SL}{{\rm SL}}

\newcommand{\ol}{\overline}

\newcommand{\dfour}[4]{\begin{footnotesize}\ensuremath{\hspace{-0.1cm}\begin{array}{ccc}&#2&\\#1\hspace{-0.2cm}&#3&\hspace{-0.2cm}#4\end{array}\hspace{-0.1cm}}\end{footnotesize}}

\newtheorem{theorem}{Theorem}[section]
\newtheorem{lemma}[theorem]{Lemma}
\newtheorem{proposition}[theorem]{Proposition}
\newtheorem{corollary}[theorem]{Corollary}
\newtheorem{definition}[theorem]{Definition}
\theoremstyle{remark}
\newtheorem{remark}[theorem]{Remark}

\begin{document}

\title[Characters of Sylow $p$-subgroups of ${}^3D_4(q^3)$]{Irreducible characters of Sylow $p$-subgroups of the Steinberg triality groups ${}^3D_4(p^{3m})$}

\author{Tung Le}

\address{Mathematics Department, North-West University, Mafikeng, South Africa}

\email{lttung96@yahoo.com}

\date{\today}




\dedicatory{Dedicated to Professor Geoffrey Robinson on his 60th birthday}

\begin{abstract}

Here we construct and count all ordinary irreducible characters of Sylow $p$-subgroups of the Steinberg triality groups ${}^3D_4(p^{3m})$.

\end{abstract}

\maketitle



\section{Introduction}
Let $\F_q$ be a finite field of order $q$ where $q$ is a power of some prime $p$.
A long standing conjecture by Higman \cite{Hig} on the number $k(U_n(q))$ of conjugacy classes of the unitriangular group $U_n(q)$ of degree $n$ over $\F_q$ is that $k(U_n(q))$ is a polynomial in $q$ with integral coefficients. The unitriangular group $U_n(q)$ is also known as a maximal unipotent subgroup of the special linear group $\SL_n(q)$. A generalization of Higman's Conjecture on the maximal unipotent subgroups $U(q)$ of other finite groups $G(q)$ of Lie type is that $k(U(q))$ is a polynomial in $q$ with integral coefficients. Here, we answer this for the Steinberg triality groups ${}^3D_4(q^3)$. Denote by $U$ a Sylow $p$-subgroup of the Steinberg triality group ${}^3D_4(q^3).$ Let $\F_{-}^\times:=\F_{-}-\{0\}$. We prove the following result.

\begin{theorem}
\label{mainthm}
The irreducible characters of $U$ are classified into five families as listed in Table \ref{tab:irrU}.
\end{theorem}

\begin{table}[!ht]
\caption{Irreducible characters of Sylow $p$-subgroups $U$ of ${}^3D_4(q^3)$}
\label{tab:irrU}
\begin{center}
\begin{tabular}{l|l|l|l|l}
\hline
\rule{0cm}{0.4cm}
Family & Notation & Parameter set & Number & Degree
\rule[-0.1cm]{0cm}{0.4cm}\\
\hline
\rule{0cm}{0.5cm}
$\FF_{6}$ & $\chi^{a,b}_{6,q^4}$ & $\F_q^\times \times \F_{q^3}$ & $(q-1)q^3$ & $q^4$
\rule[-0.2cm]{0cm}{0.4cm}\\
\hline
\rule{0cm}{0.4cm}
$\FF_{5}$ & $\chi^{a_1,a_2,b}_{5,q^3}$ & $\F_q^\times\times\F_q\times\F_{q^3}$ &$(q-1)q^4$ & $q^3$
\rule[-0.2cm]{0cm}{0.4cm}\\
\hline
\rule{0cm}{0.4cm}
$\FF_{4}^{odd}$ & $\chi^{b,a}_{4,q^3}$ & $\F_{q^3}^\times\times\F_q$ & $(q^3-1)q$ & $q^3$
\rule[-0.2cm]{0cm}{0.4cm}\\
\hline
\rule{0cm}{0.4cm}
$\FF_{4}^{even}$ & $\chi^b_{4,q^3}$ & $\F_{q^3}^\times$ &$q^3-1$ & $q^3$
\rule[-0.2cm]{0cm}{0.4cm}\\
                        & $\chi^{b,a,c_1,c_2}_{4,\frac{q^3}{2}}$ & $\F_{q^3}^\times\times\F_q^\times\times\F_2\times\F_2$ &$4(q^3-1)(q-1)$ & $q^3/2$\\
\hline
\rule{0cm}{0.4cm}
$\FF_{3}$ & $\chi^{b_1,b_2}_{3,q}$ & $\F_{q^3}^\times\times(\F_{q^3}/\F_q)$ &$(q^3-1)q^2$ & $q$
\rule[-0.2cm]{0cm}{0.4cm}\\
\hline
\rule{0cm}{0.4cm}
$\FF_{lin}$ & $\chi^{b,a}_{lin}$ & $\F_{q^3}\times\F_q$ &$q^4$ & $1$
\rule[-0.2cm]{0cm}{0.4cm}\\
\hline
\end{tabular}
\end{center}
\end{table}

It is well-known that the primes $2$ and $3$ are bad for the Chevalley groups of type $G_2$. Since the Steinberg triality groups ${}^3D_4(q^3)$ also has the Dynkin diagram of type $G_2$, we are curious that if the primes $2$ and $3$ show up as the bad primes of $^3D_4(q^3)$ in terms of the representation theory of the Sylow $p$-subgroup. Theorem \ref{mainthm} points out that only prime $2$ affects on the structure of $U$ and $k(U)$, which is compatible to the global computation on character degrees of $^3D_4(q^3)$, see \cite{Lue}.

\begin{corollary}
\label{maincoro}
If $q$ is odd then $k(U)=2q^5+2q^4-q^3-q^2-q$. Otherwise, if $q$ is even, then $k(U)=2q^5+5q^4-4q^3-q^2-4q+3$.
\end{corollary}

We approach the Sylow $p$-subgroup $U$ of 
$^3D_4(q^3)$ by its root system. The method to construct all irreducible characters of $U$ is quite elementary, mainly using Clifford theory. By studying the action of $\F_{q^3}^\times$ on its $\F_q$-hyperplane set and its $\F_p$-hyperplane set, we obtain the structures of $U$ and the factor groups $U/Z$ where $Z$ is normal in $U$ and generated by some root subgroups. 

Here are some explanations for the information in Table \ref{tab:irrU}. There are $5$ families of irreducible characters of $U$ and each row of Table \ref{tab:irrU} represents one of these families. The first column gives notation for these families of characters. Notice that the family $\FF_4^{odd}$ exists only for odd $q$, while $\FF_4^{even}$ exists only if $q$ is even. The index of this notation describes the positive root $j$-th of maximal height such that $Y_j$ is not contained in the kernels of the irreducible characters in the family, where the positive root set of ${}^3D_4$ presented in Subsection \ref{subsec:root}. The family $\FF_{lin}$ contains all linear characters of $U$. The second column of Table \ref{tab:irrU} gives the notation of irreducible characters in each family. The upper indices are the parameters to decide the uniqueness of each member in their family where $a,a_i\in\F_q$, $b,b_i\in\F_{q^3}$ and $c_i\in\F_2$. These parameters take values  from the parameter set in the third column. The lower indices show their family and degree. The fourth column lists the number of distinct irreducible characters in each family and the last column gives their degrees.

In this paper, we first present some notations of character theory, introduce the root system as well as Sylow $p$-subgroups of the Steinberg triality groups ${}^3D_4(q^3)$. Next, we establish some finite field properties which are used latter for the proof of Theorem \ref{mainthm} in Section \ref{sec:proof}.




\section{Basic Setup and Notations}
\label{sec:setup}
In this section we introduce some fundamental notations of character theory, Sylow $p$-subgroups of the Steinberg triality ${}^3D_4(q^3)$, and some finite field properties.




\subsection{Character theory}
\label{subsec:group+chars}

Let $G$ be a group. Denote $G^\times:=G-\{1\},$ $\Irr(G)$ the set of all complex irreducible characters of
$G,$ and $\Irr(G)^\times:=\Irr(G)-\{1_G\}.$ If $\chi$ is a character of $G$ and $\lambda$ a character of a subgroup $H$ of $G$, we write
$\lambda^G$ for the character induced by $\lambda$ and $\chi|_H$ for
the restriction of $\chi$ to $H$. We define $\Irr(G,\lambda):=\{\chi\in \Irr(G):(\chi,\lambda^G)>0\}$ the irreducible constituent set of $\lambda^G$. Denote the kernel of $\chi$ by $\ker(\chi):=\{g\in G:\chi(g)=\chi(1)\}$.
Furthermore, for $N\unlhd G,$ let $\Irr(G/N)$ denote the set of all irreducible characters of $G$ with $N$ in their kernels.
 For the others, our notations will be quite standard.




\subsection{Root systems of $D_4$ and ${}^3D_4$ through the graph automorphism}

\label{subsec:root}

Let $\alpha_1$, $\alpha_2$, $\alpha_3$, $\alpha_4$ be fundamental roots of the root system $\Phi$ of Lie type $D_4$. Here is the Dynkin diagram of $\Phi$, see Carter \cite[Chapter 3]{cart2}.

\begin{center}
\setlength{\unitlength}{1cm}
\begin{picture}(4,2)
\thinlines
\put(0.5,0.5){\circle*{0.17}}
\put(2,0.5){\circle*{0.17}}
\put(3.5,0.5){\circle*{0.17}}
\put(2,1.7){\circle*{0.17}}
\put(0.5,0.5){\line( 1, 0){1.5}}
\put(2,0.5){\line( 1, 0){1.5}}
\put(2,1.7){\line( 0, -1){1.2}}
\put(0.3,0.1){$\alpha_2$}
\put(1.9,0.1){$\alpha_1$}
\put(3.4,0.1){$\alpha_4$}
\put(2.2,1.7){$\alpha_3$}
\end{picture}
\end{center}

The positive roots are those roots which can be written as linear
combinations of the simple roots $\alpha_1, \alpha_2, \alpha_3,
\alpha_4$ with nonnegative coefficients and we write $\Phi_+$ for the
set of positive roots. We use the notation \dfour1121 for the
root $2\alpha_1 + \alpha_2 + \alpha_3 + \alpha_4$ and we use a
similar notation for the remaining positive roots. The $12$ positive
roots of $\Phi$ are given in Table~\ref{tab:posroots}.

\begin{table}[!ht]
\caption{Positive roots of the root system $\Phi$ of type $D_4$.}
\label{tab:posroots}

\begin{center}
\begin{tabular}{c|llll}
\hline
\rule{0cm}{0.4cm}
Height & Roots &&&
\rule[-0.1cm]{0cm}{0.4cm}\\
\hline
\rule{0cm}{0.5cm}
5 & $\alpha_{12} := $ \dfour1121 &&&
\rule[-0.2cm]{0cm}{0.4cm}\\
\hline
\rule{0cm}{0.4cm}
4 & $\alpha_{11} := $ \dfour1111 &&&
\rule[-0.2cm]{0cm}{0.4cm}\\
\hline
\rule{0cm}{0.4cm}
3 & $\alpha_8 := $ \dfour1110 & $\alpha_9 := $ \dfour0111 & $\alpha_{10} := $ \dfour1011 &
\rule[-0.2cm]{0cm}{0.4cm}\\
\hline
\rule{0cm}{0.4cm}
2 & $\alpha_5 := $ \dfour1010 & $\alpha_6 := $ \dfour0110 & $\alpha_7 := $ \dfour0011 &
\rule[-0.2cm]{0cm}{0.4cm}\\
\hline
\rule{0cm}{0.4cm}
1 & $\alpha_1:=$ \dfour0010  & $\alpha_2:=$ \dfour1000  & $\alpha_3:=$ \dfour0100   & $\alpha_4:=$ \dfour0001
\rule[-0.2cm]{0cm}{0.4cm}\\
\hline
\end{tabular}
\end{center}
\end{table}

Let $\mathfrak{L}$ be a simple complex Lie algebra with root system $\Phi$.
We choose a Chevalley basis $\{h_r | r \in \Delta\} \cup \{e_r | r
\in \Phi\}$ of $\mathfrak{L}$ such that the structure constants $N_{rs}$ in
$[e_r, e_s] = N_{rs} e_{r+s}$
on extraspecial
pairs of roots $(r, s) \in \Phi \times \Phi$ are chosen in such a way that
they are invariant under the triality automorphism $\tau$, see
\cite[Section~4.2]{cart2}, i.e.  $N_{rs} =  N_{\tau(r)\tau(s)}$.
Fix a power $p^f$ and let $D_4(p^f)$ be the Chevalley
group of type $D_4$ over the
field $\F_{p^f}$ constructed from $\mathfrak{L}$, see \cite[Section~4.4]{cart2}.
The group
$D_4(p^f) \cong P\Omega_8^+(p^f)$
is simple and is generated by the root elements $x_\alpha(t)$ for all
$\alpha \in \Phi$ and $t \in \F_{p^f}$. Let
$X_\alpha := \langle x_\alpha(t) \, | \, t \in \F_{p^f} \rangle$ be the
root subgroup corresponding to $\alpha \in \Phi$.
For positive roots, we use the abbreviation
$x_i(t) := x_{\alpha_i}(t)$, $i\in[1,12]$. The
commutators $[x_i(t), x_j(u)] = x_i(t)^{-1} x_j(u)^{-1} x_i(t) x_j(u)$
are given in Table~\ref{tab:commrelD4}. All
$[x_i(t), x_j(u)]$ not listed in this table are equal to~$1$.

\begin{table}[!ht]
\caption{Commutator relations for type $D_4$.}
\label{tab:commrelD4}
\begin{tabular}{llllll}
$\left[x_1(t), x_2(u)\right]$ & $=$ & $x_5(tu)$, \, &
$\left[x_1(t), x_3(u)\right]$ & $=$ & $x_6(tu)$, \\
$\left[x_1(t), x_4(u)\right]$ & $=$ & $x_7(tu)$, \, &
$\left[x_1(t), x_{11}(u)\right]$ & $=$ & $x_{12}(tu)$, \\
$\left[x_2(t), x_6(u)\right]$ & $=$ & $x_8(tu)$, \, &
$\left[x_2(t), x_7(u)\right]$ & $=$ & $x_{10}(tu)$,\\
$\left[x_2(t), x_{9}(u)\right]$ & $=$ & $x_{11}(tu)$, \, &
$\left[x_3(t), x_5(u)\right]$ & $=$ & $x_8(tu)$,\\
$\left[x_3(t), x_7(u)\right]$ & $=$ & $x_9(tu)$, \, &
$\left[x_3(t), x_{10}(u)\right]$ & $=$ & $x_{11}(tu)$,\\
$\left[x_4(t), x_5(u)\right]$ & $=$ & $x_{10}(tu)$, \, &
$\left[x_4(t), x_6(u)\right]$ & $=$ & $x_9(tu)$,\\
$\left[x_4(t), x_8(u)\right]$ & $=$ & $x_{11}(tu)$, \, &
$\left[x_5(t), x_9(u)\right]$ & $=$ & $x_{12}(tu)$,\\
$\left[x_6(t), x_{10}(u)\right]$ & $=$ & $x_{12}(tu)$,\, &
$\left[x_7(t), x_8(u)\right]$ & $=$ & $x_{12}(tu)$.
\end{tabular}
\end{table}

Let $U\!D_4(p^f)$ be the subgroup of $D_4(p^f)$ generated by the
elements $x_i(t)$ for $i\in[1,12]$ and $t \in \F_{p^f}$. So $U\!D_4(p^f)$ is a
maximal unipotent subgroup and a Sylow $p$-subgroup of $D_4(p^f)$.
Recall that the map
$\tau: U\!D_4(p^f) \rightarrow U\!D_4(p^f)$ defined by $\tau(x_i(t)) = x_{i\gamma}(t)$ where $\gamma$ is the permutation
$(2,3,4)(5,6,7)(8,9,10)$ induces an automorphism of $U\!D_4(p^f)$ of order $3$.
The signs in Table \ref{tab:commrelD4} are chosen to satisfy this choice of $\tau$.
Notice that $\tau$  is the restriction of a triality automorphism of the Chevalley group
$D_4(p^f)$ to $U\!D_4(p^f)$, see \cite[Proposition 12.2.3]{cart2}.

Under the action of the permutation $\gamma$ on the positive roots of $\Phi$, there are six orbits of roots as follows.
\begin{itemize}
\item[] $S_1:=\{\alpha_1=$ \dfour0010$\}$,

\item[] $S_2:=\{\alpha_2=$ \dfour1000, $\alpha_3=$ \dfour0100, $\alpha_4=$ \dfour0001$\}$,

\item[] $S_3:=\{\alpha_5=$ \dfour1010, $\alpha_6=$ \dfour0110, $\alpha_7=$ \dfour0011$\}$,

\item[] $S_4:=\{\alpha_8=$ \dfour1110, $\alpha_9=$ \dfour0111, $\alpha_{10}=$ \dfour1011$\}$,

\item[] $S_5:=\{\alpha_{11}=$ \dfour1111$\}$,

\item[] $S_6:=\{\alpha_{12}=$ \dfour1121$\}$.
\end{itemize}

For the construction of ${}^3D_4$, it requires an automorphism $\sigma=\tau\rho$ of $D_4(p^f)$ where $\rho$ is a field automorphism of $\F_{p^f}$ such that $\rho^3=1$, see \cite[Section 13.4]{cart2}. The existence of an order $3$ field automorphism $\rho$ of $\F_{p^f}$ forces $f=3m$ for some $m\in \N$. So from now on we consider the field $\F_{q^3}$. For all $t\in \F_{q^3}$, denote $\bar{t}:=\rho(t)=t^q$ and $\bar{\bar{t}}:=\rho(\bar{t})$.
By \cite[Proposition 13.6.3]{cart2}, the $\sigma$-fixed points of $U\!D_4(q^3)$ corresponding to each root orbit are as follows.
\begin{itemize}
\item[] ${Y_1} :=
\{ {y_1}(t):=x_1(t): t=\bar{t}\in \F_{q^3} \}=\{y_1(t):t\in\F_q\},$

\item[]${Y_2} :=
\{ {y_2}(t):=x_2(t)x_3(\bar{t})x_4({\bar{\bar{t}}}): t\in \F_{q^3} \},$

\item[]${Y_3} :=
\{ {y_3}(t):=x_5(t)x_6(\bar{t})x_7({\bar{\bar{t}}}): t\in \F_{q^3} \},$

\item[]${Y_4} :=
\{ {y_4}(t):=x_8(t)x_9(\bar{t})x_{10}({\bar{\bar{t}}}): t\in \F_{q^3} \},$

\item[]${Y_5} :=
\{ {y_5}(t):=x_{11}(t): t=\bar{t} \in \F_{q^3} \}=\{y_5(t):t\in\F_q\},$

\item[]${Y_6} :=
\{ {y_6}(t):=x_{12}(t): t=\bar{t} \in \F_{q^3} \}=\{y_6(t):t\in\F_q\}.$
\end{itemize}

Let $\{p_i\}_i$ be Euclidian coordinates of the roots $\{\alpha_i\}_i$. By \cite[Section 13.3.4]{cart2}, we set
$P_1:=p_1$,
$P_2:=\frac{1}{3}(p_2+p_3+p_4)$,
$P_3:=\frac{1}{3}(p_5+p_6+p_7)$,
$P_4:=\frac{1}{3}(p_8+p_9+p_{10})$,
$P_5:=p_{11}$,
$P_6:=p_{12}$.
It is easy to check that $P_3=P_1+P_2$, $P_4=P_1+2P_2$, $P_5=P_1+3P_2$, $P_6=2P_1+3P_2$. So $\{P_i: i\in [1,6]\}$ is the positive root set of ${}^3D_4$ of type $G_2$, where $P_2$ is short and $P_1$ is long.

We call each ${Y_i}$ a \textit{root subgroup} in the $\sigma$-fixed-point subgroup ${}^3D_4(q^3)$. 
It is clear that each $Y_i$ is abelian, the subgroup generated by all ${Y_i}$'s is a maximal unipotent subgroup and a Sylow $p$-subgroup of ${}^3D_4(q^3)$, denoted it by $U$.
Since ${Y_1}$, ${Y_5},$ ${Y_6}$ have order $q$ and ${Y_2}$, ${Y_3}$, ${Y_4}$ have order $q^3$, the Sylow $p$-subgroup $U$ of ${}^3D_4(q^3)$
has order $q^{3\cdot 3+3}=q^{12}$.
Using Table \ref{tab:commrelD4} we obtain the commutator relations among root subgroups ${Y_i}$ by direct computation. All $[y_i(t), y_j(u)]$ not listed bellow are equal to $1$. For all $t,u$ in the appropriate fields to root subgroups $Y_i$'s, we have
\begin{itemize}
\item[] $[{y_1}(t),{y_2}(u)]={y_3}(tu){y_4}(-tu\bar{u}){y_5}(tu\bar{u}\bar{\bar{u}}){y_6}(t^2u\bar{u}\bar{\bar{u}})$,

\item[] $[{y_2}(t),{y_3}(u)]={y_4}(t\bar{u}+\bar{t}u){y_5}(-t\bar{t}\bar{\bar{u}}-\bar{t}\bar{\bar{t}}u-\bar{\bar{t}}t\bar{u})
{y_6}(-t\bar{u}\bar{\bar{u}}-\bar{t}\bar{\bar{u}}u-\bar{\bar{t}}u\bar{u})$,

\item[] $[{y_2}(t),{y_4}(u)]={y_5}(t\bar{u}+\bar{t}\bar{\bar{u}}+\bar{\bar{t}}u)$,

\item[] $[{y_3}(t),{y_4}(u)]={y_6}(t\bar{u}+\bar{t}\bar{\bar{u}}+\bar{\bar{t}}u)$,

\item[] $[{y_1}(t),{y_5}(u)]={y_6}(tu)$.
\end{itemize}


\smallskip
From these commutator relations, it is clear that $Z(U)=Y_6,$ $Z(U/Y_6)=Y_5Y_6,$ $Z(U/Y_5Y_6)=Y_4Y_5Y_6$, $Z(U/Y_4Y_5Y_6)=Y_3Y_4Y_5Y_6$, and $U/Y_3Y_4Y_5Y_6$ is abelian of order $q^4$.
To classify all irreducible characters of $U$, we come up with the definition of almost faithful irreducible characters.

\begin{definition}
\label{def:afaithful}
Let $\chi$ be an irreducible character of a group $G$. $\chi$ is said to be {\em almost faithful} if $Z(G)\not\leq\ker(\chi)$.
\end{definition}

Due to the center series of $U$ and the inflation of irreducible characters from quotient groups, the irreducible characters $\Irr(U)$ are classified as follows.

$\FF_{6}:=\{\chi\in\Irr(U):Y_6\not\subset \ker(\chi)\}$, the set of all almost faithful irreducible characters of $U$.

$\FF_{5}:=\{\chi\in\Irr(U):Y_5\not\subset \ker(\chi) \mbox{ and } Y_6\leq \ker(\chi)\}$, the set of all almost faithful irreducible characters of $U/Y_6$.

$\FF_{4}:=\{\chi\in\Irr(U):Y_4\not\subset \ker(\chi) \mbox{ and } Y_5Y_6\leq \ker(\chi)\}$, the set of all almost faithful irreducible characters of $U/Y_5Y_6$.

$\FF_{3}:=\{\chi\in\Irr(U):Y_3\not\subset \ker(\chi) \mbox{ and } Y_4Y_5Y_6\leq \ker(\chi)\}$, the set of all almost faithful irreducible characters of $U/Y_4Y_5Y_6$.

$\FF_{lin}:=\{\chi\in\Irr(U):Y_3Y_4Y_5Y_6\leq \ker(\chi)\}$, the linear character set of $U$.

\begin{remark}
For $i\geq 3$, each $\chi\in\FF_{i}$ satisfies that $Y_i\leq Z(\chi)$ and $Y_i\not\subset \ker(\chi)$ where $Z(\chi)=\{u\in U:|\chi(u)|=\chi(1)\}$. Thus $\chi|_{Y_i}=\chi(1)\, \lambda$ for some $\lambda\in\Irr(Y_i)^\times$.
\end{remark}




\subsection{Some fundamental field results}
\label{subsec:field}

Through out this paper, for each prime power $q$, we consider the field extension $\F_{q^3}/\F_q$. Fix nontrivial linear characters $\phi:\F_q\rightarrow \C^\times$, and $\varphi:\F_{q^3}\rightarrow \C^\times$. For each $a\in \F_q$, $b\in \F_{q^3}$, define $\phi_a$ and $\varphi_b$ by $\phi_a(x):=\phi(ax)$ for all $x\in\F_q$, and $\varphi_b(y):=\varphi(by)$ for all $y\in\F_{q^3}$. Hence, $\Irr(\F_q)=\{\phi_a: a\in \F_q\}$ and $\Irr(\F_{q^3})=\{\varphi_b: b\in\F_{q^3}\}$. Recall the Frobenius map $\rho:\F_{q^3}\to \F_{q^3},\ t\mapsto t^q$, the notations $\bar{t}:=\rho(t)$ and $\bar{\bar{t}}:=\rho(\bar{t})$ for all $t\in\F_{q^3}$.

\begin{definition} \label{def:All-letters}

For each $a\in \F_q$ and $t\in \F_{q^3}$, we define

\begin{itemize}
\item[(i)] $\T_a :=\{x^p-a^{p-1}x: x\in\F_q\}$,

\item[(ii)] $\A_t:=\{t\bar{u}+\bar{t}\bar{\bar{u}}+\bar{\bar{t}}u : u\in \F_{q^3} \}$,

\item[(iii)] $\B_t:=\{t^qu+tu^q: u\in \F_{q^3}\}$.
\end{itemize}

\end{definition}
\noindent Notice that $\T_0=\F_q,$ $\A_0=\{0\}=\B_0$. Now we observe a few important properties of $\T_a$, $\A_t$ and $\B_t$.

\begin{proposition} \label{prop:mainFp}
The following are true.

\begin{itemize}

\item[(i)] $x^p-a^{p-1}x=\prod_{c\in \F_p}(x-ca).$

\item[(ii)] If $a\in \F_q^\times$, then $\T_a$ is an additive subgroup of $\F_q$ of index $p$.

\item[(iii)] For each $a\in\F_q^\times,$ there exists $b\in\F_q^\times$ such that $b\T_a=\ker(\phi).$ Furthermore, $cb\T_a=\ker(\phi)$ iff $c\in\F_p^\times.$

\item[(iv)] $\{ \T_a:a\in\F_q^\times\}=\{\ker(\phi_a):a\in \F_q^\times\}$ are all subgroups  of index $p$ in $\F_q.$
\end{itemize}

\end{proposition}
\begin{proof}
Part (i) is clear by checking their solutions directly. For each $a\in \F_q^\times$ we define $h_a(x):=x^p-a^{p-1}x$ for all $x\in \F_q$.
It is clear that $\T_a=\im(h_a)$ and $h_a$ is an $\F_p$-homomorphism with $\ker(h_a)=a\F_p$ by part (i). Thus part (ii) holds. Since part (iv) clearly follows part (iii), it suffices to show part (iii).

Let  $\mathfrak{S}:=\{\T_a: a\in\F_q^\times\}.$ For all $a,y\in\F_q^\times,$ we have $y^p\T_a=\T_{ay}$ since
\[y^ph_a(x)=(xy)^{p}-(ay)^{p-1}(xy)=h_{ay}(xy)\in\im(h_{ay})=\T_{ay}.\]
So $\F_q^\times$ acts transitively on the set $\mathfrak{S}.$ Since $\mathfrak{S}\neq \emptyset$ and $\F_p^\times\leq \Stab_{\F_q^\times}(\T_a)$, we have $0<|\mathfrak{S}|\leq \frac{q-1}{p-1}.$
It is enough to show that $\Stab_{\F_q^\times}(\T_a)=\F_p^\times$ by computing $|\mathfrak{S}|.$

Since the left multiplication action of $\F_q^\times=\{y^p:y\in\F_q^\times\}$ on the $\F_p$-hyperplane set of $\F_q$ is transitive and the number of $\F_p$-hyperplanes of $\F_q$ is $\frac{q-1}{p-1}$ which equals $|\mathfrak{S}|$, the claim holds.
\end{proof}

\begin{lemma} \label{lem:Frob-cubed}

For each $t\in \F_{q^3}^\times$, $\A_t=\F_q$.

\end{lemma}
\begin{proof}
It is clear that $\A_t\subset \F_q$ since its elements are $\rho$-invariant. 
So $|\A_t|\leq q$. It suffices to show that $|\A_t|=q$.

For each $a\in \F_q$, the equation $\bar{\bar{t}}u+tu^q+\bar{t}u^{q^2}=a$ has at most $q^2$ solutions for $u$ in $\F_{q^3}$. Therefore, when $u$ runs all over $\F_{q^3}$, $\A_t$ has at least $|\F_{q^3}|/q^2=q$ elements. This completes the proof.
\end{proof}

For each $t\in \F_{q^3}^\times$, we define an $\F_q$-homomorphism $f_t:\F_{q^3}\rightarrow \F_{q^3}$ given by $f_t(u):=t^qu+tu^q$ for all $u\in \F_{q^3}$. By Definition \ref{def:All-letters} (iii), $\B_t=\im(f_t)$.

\begin{lemma} \label{lem:odd-even_3D4}

For all $t\in\F_{q^3}^\times$, the following hold.

\begin{itemize}
\item[(i)] $vx\in \B_t$ for all $v\in \B_t$ and $x\in\F_q$. Hence, $\B_{tx}=\B_t=x\B_t$ for all $x\in\F_q^\times$.

\item[(ii)] If $q$ is odd, then $\B_t=\F_{q^3}$.

\item[(iii)] If $q$ is even, then $\ker(f_t)=t \F_q$ and $\B_t\leq (\F_{q^3},+)$ of order $q^2$.
\end{itemize}

\end{lemma}

\begin{proof} Part (i) is clear by the $\F_q$-homomorphism property of $f_t$. Thus, it suffices to show parts (ii) and (iii) by finding the kernel of $f_t$.

We have $t^qu+tu^q=t^qu(1 + (t^{-1}u)^{q-1})$. So $f_t(u)=0$ iff $u=0$ or $(t^{-1}u)^{q-1}=-1$. Now we divide the proof into two cases.

Case $q$ odd: Since $\F_{q^3}^\times$ is cyclic of order $q^3-1=(q-1)(q^2+q+1)$ and $(q^2+q+1)$ is odd, there is no $z\in \F_{q^3}^\times$ such that $z^{q-1}$ has order 2. So the equation $(t^{-1}u)^{q-1}=-1$ has no solution for $u\in \F_{q^3}$. This shows that $\ker(f_t)=\{0\}$ and $\im(f_t)=\F_{q^3}$.

Case $q$ even: The equation $(t^{-1}u)^{q-1}=1$ implies that $t^{-1}u\in \F_q^\times$, i.e. $u\in t \F_q^\times$. It is clear that $t\F_q\subset \ker(f_t)$. So $\ker(f_t)=t\F_q$ and $|\B_t|=q^2$.
\end{proof}

\begin{lemma} \label{lem:char2-3D4}

If $q$ is even, the following hold.

\begin{itemize}

\item[(i)] $\ker(f_1)=\F_q$ and $\F_{q^3}=\ker(f_1) \oplus \B_1$.

\item[(ii)] $\B_t=t^{q+1}\B_1$ for all $t\in\F_{q^3}^\times$.

\item[(iii)] $\{\B_t:t\in\F_{q^3}^\times\}$ is the $\F_q$-hyperplane set of $\F_{q^3}$. So $|\B_t\cap \B_r|=q$ iff $t\not \in r \F_q$.

\item[(iv)] There exists uniquely $t\in \F_{q^3}^\times$, up to a scalar of $\F_q^\times$, such that $\B_t \subset \ker(\varphi)$.

\end{itemize}

\end{lemma}

\begin{proof}
(i) Recall that $\B_1=\{u+u^q: u\in\F_{q^3}\}$ and $\ker(f_1)=\F_q$ by Lemma \ref{lem:odd-even_3D4} (iii). Suppose that $\ker(f_1)\cap \B_1$ is nontrivial, i.e. there exist $x\in \F_q^\times$ and $u\in \F_{q^3}^\times$ such that $x=u+u^q$.

If $u\in \F_q$, then $x=u+u^q=2u=0$ since $\chr(\F_q)=2$, contrary to $x\in\F_q^\times$.
Notice that $\F_{q^2}\not\subset \F_{q^3}$. If $u\not\in\F_q$, from $u^{q^2}+u^q+u\in \F_q$, we have $u+u^q\not\in \F_q$, contrary to $x=u+u^q\in \F_q$. So $\ker(f_1)\cap \B_1=\{0\}$.

By Lemma \ref{lem:odd-even_3D4} (iii) with $|\ker(f_1)|=q$ and $|\B_1|=q^2$, we have $\F_{q^3}=\ker(f_1)\oplus \B_1$.

\smallskip
(ii) From $t^qu+tu^q=t^{q+1}(ut^{-1}+(ut^{-1})^q)$ and $\{(ut^{-1})+(ut^{-1})^q:u\in\F_{q^3}\}=\B_1$, the claim is clear.

\smallskip
(iii) Since $q$ is even and $q^3-1=(q-1)(q^2+q+1)$, $\gcd(q+1,q^3-1)=1$ and $\{t^{q+1}:t\in\F_{q^3}^\times\}=\F_{q^3}^\times$. By part (ii), $\F_{q^3}^\times$ acts transitively on the set of all $\B_t$'s. However, the left multiplication action of $\F_{q^3}^\times$ is also transitive on the $\F_q$-hyperplane set of $\F_{q^3}$. This shows that $\{\B_t:t\in\F_{q^3}^\times\}$ is the $\F_q$-hyperplane set of $\F_{q^3}$.

By part (ii), $\B_t=\B_r$ iff $(tr^{-1})^{q+1}\in\Stab_{\F_{q^3}^\times}(\B_1)$. The transitivity of $\F_{q^3}^\times$ on the $\F_q$-hyperplane set, whose cardinality is $\frac{q^3-1}{q-1}$, implies that $|\Stab_{\F_{q^3}^\times}(\B_1)|=q-1$. By Lemma \ref{lem:odd-even_3D4} (i), $\Stab_{\F_{q^3}^\times}(\B_1)=\F_q^\times$. It is clear that $(tr^{-1})^{q+1}\in\F_q^\times$ iff $t\in r\F_q^\times$.
By the property of $\F_q$-hyperplanes in $\F_{q^3}$, the rest of the claim is clear.

\smallskip
(iv) The uniqueness follows by part (iii) since $\B_t+\B_s=\F_{q^3} \supsetneq \ker(\varphi)$ for any $\B_t\neq\B_s$.
It is clear that $\B_1$ is contained in $\ker(\varphi_b)$ for some $\varphi_b\in\Irr(\F_{q^3})^\times$. The existence follows by the transitivity of the left multiplication action of $\F_{q^3}^\times$ on the $\F_p$-hyperplane set of $\F_{q^3}$, also known as $\{\ker(\varphi_r):\varphi_r\in\Irr(\F_{q^3})^\times\}$.
\end{proof}




\section{Characters of Sylow $p$-subgroups of the triality groups ${}^3D_4(q^3)$}
\label{sec:proof}

Let $\chi\in\Irr(U)$. We shall prove Theorem \ref{mainthm} by constructing $\chi$ through each family from $\FF_{6}$ to $\FF_{5},...,\FF_{lin}$ as discussed in Subsection \ref{subsec:root}. We recall the commutator relations computed in Subsection \ref{subsec:root}.
\begin{itemize}
\item[] $[{y_1}(t),{y_2}(u)]={y_3}(tu){y_4}(-tu\bar{u}){y_5}(tu\bar{u}\bar{\bar{u}}){y_6}(t^2u\bar{u}\bar{\bar{u}})$,

\item[] $[{y_2}(t),{y_3}(u)]={y_4}(t\bar{u}+\bar{t}u){y_5}(-t\bar{t}\bar{\bar{u}}-\bar{t}\bar{\bar{t}}u-\bar{\bar{t}}t\bar{u})
{y_6}(-t\bar{u}\bar{\bar{u}}-\bar{t}\bar{\bar{u}}u-\bar{\bar{t}}u\bar{u})$,

\item[] $[{y_2}(t),{y_4}(u)]={y_5}(t\bar{u}+\bar{t}\bar{\bar{u}}+\bar{\bar{t}}u)$,

\item[] $[{y_3}(t),{y_4}(u)]={y_6}(t\bar{u}+\bar{t}\bar{\bar{u}}+\bar{\bar{t}}u)$,

\item[] $[{y_1}(t),{y_5}(u)]={y_6}(tu)$.
\end{itemize}

\subsection{Family $\FF_{6}$ where $\chi$ is almost faithful}
\label{subsec:F6}
Let $T:=Y_1Y_3Y_4Y_5Y_6$ and $V:=Y_4Y_5Y_6$. It is easy to check that $V$ is abelian, $V\lhd T\lhd U$ and $Z(T)=Y_6$. By Lemma \ref{lem:Frob-cubed} and Clifford theory with the transversal $Y_1Y_3$ of $V$ in $T$, all $\lambda\in \Irr(V)$ such that $\lambda|_{Y_6}\neq 1_{Y_6}$ satisfy $\lambda^T\in\Irr(T)$ of degree $q^4$. Thus all almost faithful irreducible characters of $T$ have degree $q^4$. Since $\chi|_T$ decomposes into sum of almost faithful irreducible characters, 
$\chi(1)\geq q^4$. Since $\lambda^T|_{Y_4Y_5}$ is the regular character of the abelian group $Y_4Y_5$, $\chi|_V$ has a linear constituent $\theta$ such that $\theta|_{Y_5}=1_{Y_5}$.

Let $H:={Y_6}{Y_5}{Y_4}{Y_2}$. Clearly that $H\leq U$ and $Z(H)=Y_5Y_6$. By the existence of $\theta$, let $\xi$ be an irreducible constituent of $\chi$ restricted to $H$ such that $\xi|_{Y_5}=\xi(1)1_{Y_5}$. Since $Y_5\leq \ker(\xi)$, $\xi$ can be considered as a character of $H/Y_5\cong {Y_6}{Y_4}{Y_2}$, which is abelian. So $\xi$ is linear. We shall show that $\xi^{U}\in \Irr(U)$, which implies that $\chi$ is the unique constituent in $\Irr(U,\xi)$, i.e. $\chi=\xi^{U}$. Since $H\cap T=V$ and $U=HT$, by Mackey formula for the double coset $H\backslash U/T$ we have $\xi^U|_T={\xi|_{H\cap T}}^T\in\Irr(T)$. Thus $\xi^U\in\Irr(U)$ as claimed. So all almost faithful irreducibles 
have degree $q^4$.

By the extension of each $\lambda^T$ to $U$ through $Y_2$, we obtain $(q-1)q^3$ almost faithful irreducible characters of degree $q^4$, parameterized by $(Y_6^\times, Y_2)\cong(\F_q^\times,\F_{q^3})$ and denoted by $\chi_{6,q^4}^{a,b}$.

\subsection{Family $\FF_{5}$ where ${Y_6}\subset \ker(\chi)$ and ${Y_5}\not\subset \ker(\chi)$}
\label{subsec:F5}
We study the quotient group $\bar{U}:=U/{Y_6}$. Abusing terminology slightly we call the image under the natural projection of a root group of $U$, a root group of  $\bar U$. Recall that the commutator relations in $\bar U$ are as follows.
\begin{itemize}
\item[] $[{y_1}(t),{y_2}(u)]={y_3}(tu){y_4}(-tu\bar{u}){y_5}(tu\bar{u}\bar{\bar{u}})$,

\item[] $[{y_2}(t),{y_3}(u)]={y_4}(t\bar{u}+\bar{t}u){y_5}(-t\bar{t}\bar{\bar{u}}-\bar{t}\bar{\bar{t}}u-\bar{\bar{t}}t\bar{u})$,

\item[] $[{y_2}(t),{y_4}(u)]={y_5}(t\bar{u}+\bar{t}\bar{\bar{u}}+\bar{\bar{t}}u)$.
\end{itemize}

Let $\lambda$ be a linear character of the abelian normal subgroup $H:={Y_5}{Y_4}{Y_3}{Y_1}\lhd \bar U$.
By Lemma \ref{lem:Frob-cubed}, for each ${y_2}(t)\in Y_2^\times$, there exists some ${y_4}(u)\in {Y_4}$ such that $[{y_4}(u),{y_2}(t)] \not \in \ker(\lambda)$, i.e. ${}^{y_2(t)}\lambda({y_4}(u))\neq \lambda({y_4}(u))$. So the inertia group $I_{\bar U}(\lambda)= H$. By Clifford theory, $\lambda^{\bar{U}}\in \Irr(\bar U)$.

Since $q-1$ nontrivial linear characters of $Y_5$ extend to $(q-1)q^7$ linear characters of $H$, by the above argument we obtain $(q-1)q^4$ irreducibles 
of degree $q^3$ in this family, parameterized by $(Y_5^\times,Y_1,Y_3)\cong (\F_q^\times,\F_q,\F_{q^3})$ and denoted by $\chi_{5,q^3}^{a_1,a_2,b}$.

\subsection{Family $\FF_{4}$ where ${Y_5}{Y_6}\subset \ker(\chi)$ and ${Y_4}\not\subset \ker(\chi)$}
\label{subsec:F4}
We study the quotient group  $\bar U:=U/{Y_6}{Y_5}$. Recall that the commutator relations in $\bar{U}$ are as follows.
\begin{itemize}
\item[] $[{y_1}(t),{y_2}(u)]={y_3}(tu){y_4}(-tu\bar{u})$,

\item[] $[{y_2}(t),{y_3}(u)]={y_4}(t\bar{u}+\bar{t}u)$.
\end{itemize}

Let $\lambda$ be a linear character of the abelian normal subgroup $H:={Y_4}{Y_3}{Y_1}\lhd \bar U$ such that $\lambda|_{{Y_4}}\neq 1_{Y_4}$.
Since ${Y_2}$ is a transversal of $H$ in $\bar{U}$, to obtain the inertia group $I_{\bar{U}}(\lambda)$, we find all ${y_2}(t)\in {Y_2}$ such that ${}^{{y_2}(t)}\lambda=\lambda$. 
First we compute the stabilizer $\Stab_{{Y_2}}(\lambda|_{{Y_3}})$. 
We have
\[\begin{array}{ll}
{}^{{y_2}(t)}\lambda({y_3}(u))&=\lambda({y_3}(u)^{{y_2}(t)})\\
&=\lambda({y_3}(u))\lambda([{y_3}(u),{y_2}(t)])\\
&=\lambda({y_3}(u))\lambda({y_4}(-(t\bar{u}+\bar{t}u))).
\end{array}\]

By Lemma \ref{lem:odd-even_3D4} (ii) with $q$ odd, the inertia group $I_{\bar{U}}(\lambda)=H$. By Clifford theory, $\lambda^{\bar{U}}\in \Irr(\bar{U})$. Hence there are $(q^3-1)q$ irreducible characters of degree $q^3$ in this family, parameterized by $(Y_4^\times,Y_1)\cong (\F_{q^3}^\times,\F_q)$ and denoted by $\chi_{4,q^3}^{b,a}\in\FF^{odd}_{4}$.

\smallskip
Now we observe the case where $q$ is even. Fix $\lambda|_{{Y_4}}=\varphi_{c_4}$ where $c_4\in\F_{q^3}^\times$ as described in Subsection \ref{subsec:field}. We shall find the orbit of $\lambda$ under the action of $Y_2$.
Recall the function $f_t(u):=t^qu+tu^q$ defined right before Lemma \ref{lem:odd-even_3D4}. We have
\[{}^{{y_2}(t)}\lambda({y_3}(u))=\lambda({y_3}(u))\lambda({y_4}((t\bar{u}+\bar{t}u)))=\varphi(c_3u)\, \varphi_{c_4}(f_t(u)).\]
So ${}^{y_2(t)}\lambda|_{Y_3}=\lambda|_{Y_3}$ iff $\varphi_{c_4}(f_t(u))=1$ for all $u\in\F_{q^3}$, i.e. $\im(f_t)\subset \ker(\varphi_{c_4})$. By Lemma \ref{lem:char2-3D4} (iv), there exists unique $t_0\in\F_{q^3}^\times$, up to a scalar of $\F_q^\times$, such that $\im(f_{t_0})\subset \ker(\varphi_{c_4})$. By Lemma \ref{lem:char2-3D4} (iii) $St_2:=\Stab_{Y_2}(\lambda|_{Y_3})=\{y_2(rt_0):r\in\F_q\}$. Thus, under the action of $Y_2$, $\Irr(Y_3)$ decomposes into $q$ orbits; each has size $q^2$.

Let $\lambda|_{Y_3}=\varphi_{c_3}$, $c_3\in\F_{q^3}$. Now we compute $St:=\Stab_{Y_2}(\lambda)=\Stab_{St_2}(\lambda|_{Y_1})$. 
\[\begin{array}{ll}
{}^{{y_2}(rt_0)}\lambda({y_1}(v))
&=\lambda({y_1}(v))\lambda([{y_1}(v),{y_2}(rt_0)])\\
&=\lambda({y_1}(v))\lambda({y_3}(vrt_0){y_4}(-vrt_0\overline{rt_0}))\\
&=\lambda({y_1}(v))\varphi(c_3vrt_0+c_4vr^2t_0^{q+1})\\
&=\lambda({y_1}(v))\varphi_{c_4}(vrt_0(c_3c_4^{-1}+rt_0^q)).
\end{array}\]
So ${}^{{y_2}(rt_0)}\lambda|_{Y_1}=\lambda|_{Y_1}$ iff $\varphi_{c_4}(vrt_0(c_3c_4^{-1}+rt_0^q))=1$ for all $v\in\F_q$.
We find $r\in\F_q^\times$ with parameter $c_3\in \F_{q^3}$ to obtain that $vrt_0(c_3c_4^{-1}+rt_0^q)\in \ker(\varphi_{c_4})$
for all $v\in \F_q$. Notice that if $c_3c_4^{-1}t_0^{-q}\in\F_q^\times$ then it is a nontrivial solution for $r$. We shall see that $q-1$ values of $c_3\in\F_{q^3}^\times$ such that $c_3c_4^{-1}t_0^{-q}\in\F_q^\times$  correspond to $q-1$ orbits of $\lambda$ whose $\Stab_{Y_1}(\lambda)$ has order $2$, and this is the unique nontrivial solution for $r\in\F_q^\times$ in each orbit.
Now we suppose that $c_3c_4^{-1}t_0^{-q}\not\in \F_q^\times$. So $c_3c_4^{-1}+rt_0^q\neq 0$ for all $r\in\F_q^\times$.
For each $r\in\F_q^\times$, let $T(r,c_3):=\{vrt_0(c_3c_4^{-1}+rt_0^q): v\in\F_q\}$ be a one-dimensional $\F_q$-subspace of $\F_{q^3}$.
Assume that $T(r,c_3)\subset \ker(\varphi_{c_4})\subset \F_{q^3}$.

If $T(r,c_3)\cap \B_{t_0}=\{0\}$, then $\F_{q^3}=T(r,c_3)\oplus\B_{t_0}\subset \ker(\varphi_{c_4})$, a contradiction. 
Therefore, $T(r,c_3)\cap \B_{t_0}$ is nontrivial. By Lemma \ref{lem:odd-even_3D4} (i), we have $T(r,c_3)\subset \B_{t_0}$. So
$rt_0(c_3c_4^{-1}+rt_0^q)\in \B_{t_0}$. By Lemma \ref{lem:char2-3D4} (ii) with $B_{t_0}=t_0^{q+1}\B_1$, there exists $y\in \B_{1}$ such that $rt_0(c_3c_4^{-1}+rt_0^q)=t_0^{q+1}y$. Solve this equation for $c_3$ with parameter $r\in\F_q^\times$, we have $c_3\in\{c_4t_0^q(r^{-1}y+r):r\in \F_q^\times,y\in \B_1\}=:I_3$. 

We claim $|I_3|=(q-1)q^2$ by proving that if $r^{-1}y+r=s^{-1}z+s$ for some $r,s\in \F_q^\times$ and $y,z\in\B_1$ then $r=s$ and $y=z$. Since $\B_1$ is an $\F_q$-vector space, $r^{-1}y+s^{-1}z\in \B_1$. From $r^{-1}y+r+s^{-1}z+s=0\in \B_1$, we have $r+s\in\B_1$. By Lemma \ref{lem:char2-3D4} (i) with $\F_{q^3}=\F_q \oplus \B_1$, we have $r+s\in\F_q\cap \B_1=\{0\}$, i.e. $r=s$.

This shows that each $c_3\in I_3$ determines uniquely $r_0\in\F_q^\times$ and $y\in\B_1$ such that $c_3=c_4t_0^q(r_0^{-1}y+r_0)$.
Notice that $r^{-1}y+r\in\F_q^\times$ iff $y=0$ where $r\in\F_q^\times$ and $y\in\B_1$. The above argument also confirms our discussion on $q-1$ values of $c_3\in\F_{q^3}^\times$ such that $c_3c_4^{-1}t_0^{-q}=r_0\in\F_q^\times$ which is the unique nontrivial solution of $r$, and the $q^2$-size orbit of $\lambda$ with this $r_0$  is given by obtaining all $c_3\in I_3$ from running all $y\in\B_1$.

These $(q-1)q^3$ linear characters $\lambda$ of $H$ corresponding to $c_3\in I_3^\times$ (counting $\lambda|_{Y_1}$ for other $q$ copies) have the stabilizer $St=\Stab_{{Y_2}}(\lambda)=\{1,{y_2}(r_0t_0)\}$ and $I_{\bar{U}}(\lambda)=HSt$.
It is easy to see that $[HSt,HSt]\leq \ker(\lambda)$. Thus, these $(q-1)q^3$ linear characters extend to $2(q-1)q^3$ linear characters (of their inertia groups) and induce irreducibly to $\bar{U}$, which gives $2(q-1)q^3/(q^3/2)=4(q-1)$ irreducible characters of degree $q^3/2$.
Together with all $c_4\in\F_{q^3}^\times$, we obtain $4(q^3-1)(q-1)$ irreducible characters of degree $q^3/2$ parameterized by $(Y_4^\times,Y_3^*,Y_2^*,Y_1^*)\cong (\F_{q^3}^\times,\F_q^\times,\F_2,\F_2)$ and denoted by $\chi_{4,\frac{q^3}{2}}^{b,a,c_1,c_2}\in\FF^{even}_{4}$, which is proven bellow in details.

The other $q^3$ linear characters $\lambda$ corresponding to $c_3\in\F_{q^3}-I_3$ have $\Stab_{Y_2}(\lambda)=\{1\}$. This shows that $Y_2$ acts transitively on these $q^3$ linear characters and their inertia groups equal $H$. (Notice that we can choose $\lambda$ with $c_3=0$.) So we obtain $q^3-1$ irreducible characters of $\bar{U}$ of degree $q^3$, parameterized by $(Y_4^\times)\cong (\F_{q^3}^\times)$ and denoted by $\chi_{4,q^3}^{b}\in\FF^{even}_{4}$.

\subsubsection{\em Proof of the parametrization $(Y_4^\times,Y_3^*,Y_2^*,Y_1^*)\cong (\F_{q^3}^\times,\F_q^\times,\F_2,\F_2)$ of $\chi_{4,q^3/2}^{b,a,c_1,c_2}$}

Fix $c_4\in\F_{q^3}^\times$. The parameter $Y_3^*\cong\F_q^\times$ corresponds to $q-1$ orbits of $\Irr(Y_3)$ satisfying $c_3\in I_3$ under the action of $Y_2$. 
It suffices to show the parametrization of $(Y_2^*,Y_1^*)\cong (\F_2,\F_2)$. Fix $c_3\in I_3$, let $\mu\in \Irr(Y_4Y_3)$ such that $\mu|_{Y_i}=\varphi_{c_i}$, $i=3,4$. By the above proof, $\mu$ is extendable to $K_1:=I_{\bar{U}}(\lambda)=HSt$  where $H=Y_4Y_3Y_1$ and $St
=\{1,y_2(r_0t_0)\}$. Let $\eta_1$, $\eta_2$ be two extensions of $\eta$ to $K_1$. Moreover, $\mu$  extends to $K:=Y_4Y_3St_2$ where $St_2=\Stab_{Y_2}(\lambda|_{Y_3})=\{1,y_2(rt_0):r\in\F_q\}$. Let $\theta$ be an extension of $\mu$ to $K$ and $S_1:=\Stab_{Y_1}(\theta)$.
An easy way to read this is described in the following diagram.
\[
\begin{array}{rlcrl}
     Y_4Y_3:&           & \mu   &      &\\
            & \swarrow  &               &\searrow& \\
H=Y_4Y_3Y_1:&\lambda   &        & \theta  & :Y_4Y_3St_2= K\\
            &\downarrow  &      &\downarrow&\\
K_1=HSt:    &\eta_i  &          & \gamma     & :KS_1=K_2\\
            & \searrow  &       &\swarrow &\\
 \bar{U}:   &  &\eta_i^{\bar{U}},\gamma^{\bar{U}}&\in&\Irr(\bar{U})
\end{array}
\]

We shall show the follows.
\begin{itemize}
\item[(i)] $|S_1|{=}2$ and $\theta$ extends to $K_2{:=}I_{\bar{U}}(\theta){=}KS_1$. Call $\gamma$ an extension of $\theta$ to $K_2$.
\item[(ii)] $(\eta_i^{\bar{U}},\gamma^{\bar{U}})=1$ iff $\eta_i|_{St}=\gamma|_{St}$ and $\eta_i|_{S_1}=\gamma|_{S_1}$ for $i=1,2$.
\item[(iii)] $(\eta_1^{\bar{U}},\eta_2^{\bar{U}})=1$ iff $\eta_1|_{St}=\eta|_{St}$ and $\eta|_{S_1}=\eta|_{S_1}$.
\end{itemize}

\begin{proof}
(i) Notice that $\mu$ extends to $K$ because $[K,K]\leq\ker(\mu)$. Let $T_2$ be a transversal of $St_2$ in $Y_2$. It is clear that $T_2Y_1$ is a transversal of $K$ in $\bar{U}$. Since $Y_4Y_3Y_1$ is abelian, it is enough to find $\Stab_{Y_1}(\theta|_{St_2})$. For $t,r\in\F_q$, we have
\[\begin{array}{ll}
{}^{y_1(t)}\theta(y_2(rt_0))&=\theta(y_2(rt_0)[y_2(rt_0),y_1(t)])\\
&=\theta(y_2(rt_0)y_3(trt_0)y_4(trt_0\ol{rt_0}))\\
&=\theta(y_2(rt_0))\varphi(c_3trt_0+c_4tr^{q+1}t_0^{q+1})\\
&=\theta(y_2(rt_0))\varphi_{c_4}(trt_0(c_3c_4^{-1}+rt_0^{q})).
\end{array}\]
To obtain $\Stab_{Y_1}(\theta)\neq\{1\}$ we find $t\in\F_q^\times$ such that 
$trt_0(c_3c_4^{-1}+rt_0^{q})\in\ker(\varphi_{c_4})$ for all $r\in\F_q$. By the uniqueness of $(c_3,r_0,y)\in I_3\times\F_q^\times\times\B_1$ with $c_3=c_4t_0^q(r_0^{-1}y+r_0)$, it requires that $trt_0^{q+1}(r_0^{-1}y+r_0+r)\in\ker(\varphi_{c_4})$ for all $r\in\F_q$. By Lemma \ref{lem:odd-even_3D4} (i) and Lemma \ref{lem:char2-3D4} (ii) with $trr_0^{-1}\in\F_q$ and $t_0^{q+1}y\in t_0^{q+1}\B_1=\B_{t_0}\subset \ker(\varphi_{c_4})$, it suffices to find $t\in\F_q^\times$ such that $tt_0^{q+1}r(r_0+r)\in\ker(\varphi_{c_4})$ for all $r\in\F_q$.

If $tt_0^{q+1}r(r_0+r)\in\B_{t_0}\subset\ker(\varphi_{c_4})$, then by Lemma \ref{lem:odd-even_3D4} (i) with $t,r,r_0\in\F_q^\times$ we have $t_0^{q+1}\in\B_{t_0}.$ By Lemma \ref{lem:char2-3D4} (ii) with $\B_{t_0}=t_0^{q+1}\B_1$, it implies that $1\in\B_1$ which contradicts Lemma \ref{lem:char2-3D4} (i). Thus, $tt_0^{q+1}r(r_0+r)\not\in\B_{t_0}$ for all $r\in\F_q^\times$.

Let $A_t:=\{tr(r_0+r):r\in\F_q\}$ for $t\in\F_q^\times$. Notice that $A_t$ is an $\F_2$-hyperplane of $\F_q$ due to the $\F_2$-homomorphism $h_{r_0}:\F_q\to\F_q$, $r\mapsto r_0r+r^2$. If $t_0^{q+1}A_t\subset \ker(\varphi_{c_4})$, using the above argument we have $\ker(\varphi_{c_4})=\B_{t_0}\oplus t_0^{q+1}A_t$, an $\F_2$-hyperplane of $\F_{q^3}$. By Lemma \ref{lem:char2-3D4} (ii) we have $\ker(\varphi_{c_4})=t_0^{q+1}(\B_1\oplus A_t)$. Notice that $\F_q^\times$ acts transitively on the $\F_2$-hyperplanes of $\F_q$ and the number of $\F_2$-hyperplanes of $\F_q$ is $q-1$. So there exists uniquely $t_1\in\F_q^\times$ such that $\ker(\varphi_{c_4})=t_0^{q+1}(\B_1\oplus A_{t_1})$, which shows that $S_1=\Stab_{Y_1}(\theta)=\{1,y_1(t_1)\}$ as claimed. Here, the extension of $\theta$ to $I_{\bar{U}}(\theta)=KS_1$ is clear by $[KS_1,KS_1]\leq\ker(\theta)$.

\smallskip
(ii) Let $\gamma$ be an extension of $\theta$ to $K_2=KS_1$ and $\eta$ an extension of $\lambda$ to $K_1=HSt$. We shall show that  $(\eta^{\bar{U}},\gamma^{\bar{U}})=1$ iff $\eta|_{St}=\gamma|_{St}$ and $\eta|_{S_1}=\gamma|_{S_1}$. Notice that both $\eta^{\bar{U}},\gamma^{\bar{U}}\in\Irr(\bar{U})$ by Clifford theory.

Since $K_1\lhd \bar{U}=K_1K_2$, by Mackey formula for the double coset $K_1\backslash \bar{U}/K_2$ and Frobenius reciprocity  we have
\[\begin{array}{ll}
(\eta^{\bar{U}},\gamma^{\bar{U}})&=(\eta^{\bar{U}}|_{K_2},\gamma)\\
&=({\eta|_{K_1\cap K_2}}^{K_2},\gamma)\\
&=(\eta|_{K_1\cap K_2},\gamma|_{K_1\cap K_2}).
\end{array}\]
Since $K_1\cap K_2=Y_4Y_3StS_1$ and both $\eta,\gamma$ are linear, the claim holds.

\smallskip
(iii) Choosing $\gamma\in\Irr(K_2)$ such that $\gamma|_{K_1\cap K_2}=\eta_1|_{K_1\cap K_2}$, we have $\eta_1^{\bar{U}}=\gamma^{\bar{U}}$ by part (ii). Again by part (ii) we have $(\eta_2^{\bar{U}},\gamma^{\bar{U}})=1$ iff $\eta_2|_{St}=\gamma|_{St}$ and $\eta_2|_{S_1}=\gamma|_{S_1}$, which completes the proof.
\end{proof}

\subsection{Family $\FF_{3}$ where ${Y_4}{Y_5}{Y_6}\subset \ker(\chi)$ and ${Y_3}\not\subset \ker(\chi)$}
\label{subsec:F3}
Set  $\bar U:=U/{Y_6}{Y_5}{Y_4}$. Recall that the unique commutator relation in $\bar{U}$ is $[{y_1}(t),{y_2}(u)]={y_3}(tu)$. Let $\lambda$ be a linear character of the abelian normal subgroup $H:={Y_3}{Y_2}\lhd \bar U$ such that $\lambda|_{{Y_3}}=\varphi_{c_3}$ for some $c_3\in\F_{q^3}^\times$. Clearly ${Y_1}$ is a transversal of $H$ in $\bar{U}$.

For each ${y_1}(t)\in {Y_1}$ and all ${y_2}(u)\in {Y_2}$, we have
\[\begin{array}{ll}
{}^{{y_1}(t)}\lambda({y_2}(u))
&=\lambda({y_2}(u))\lambda([{y_2}(u),{y_1}(t)])\\
&=\lambda({y_2}(u))\lambda({y_3}(tu))\\
&=\lambda({y_2}(u))\varphi_{c_3}(tu).
\end{array}\]
So ${}^{{y_1}(t)}\lambda=\lambda$ iff $\varphi_{c_3}(tu)=1$ for all $u\in \F_{q^3}$, clearly iff $t=0$. Thus, the inertia group
$I_{\bar{U}}(\lambda)=H$. By Clifford theory, $\lambda^{\bar{U}}\in \Irr(\bar{U})$ of degree $q$.

So $\FF_{3}$ contains $(q^3-1)q^2$ irreducible characters of degree $q$, parameterized by $(Y_3^\times,Y_2^*)\cong (\F_{q^3}^\times,\F_{q^3}/\F_q)$ and denoted by $\chi_{3,q}^{b_1,b_2}$, where $Y_2^*$ is the representative set of $q^2$ orbits of $\Irr(Y_2)$ under the action of $Y_1$ for a fixed $c_3\in\F_{q^3}^\times$.

\subsection{Family $\FF_{lin}$ where ${Y_3}{Y_4}{Y_5}{Y_6}\subset \ker(\chi)$}
\label{subsec:Flin}
Set $\bar{U}:=U/{Y_6}{Y_5}{Y_4}{Y_3}$. Since $\bar{U}$ is abelian, this family is the set of all linear characters of $\bar{U}$. Here we obtain $q^4$ linear characters parameterized by $(Y_2,Y_1)\cong (\F_{q^3},\F_q)$ and denoted by $\chi_{lin}^{b,a}$.

\end{document}